\documentclass[dvipdfmx,12pt]{amsart}

\usepackage{amsmath,amsthm,amssymb,mathdots,shuffle}
\usepackage{setspace}
\usepackage{mathtools}
\usepackage{color}
\usepackage[capitalize]{cleveref}
\usepackage{autonum}
\allowdisplaybreaks
\newtheorem{theorem}{Theorem}[section]

\newtheorem{lemma}[theorem]{Lemma}

\theoremstyle{remark}
\newtheorem{remark}[theorem]{Remark}

\numberwithin{equation}{section}

\newcommand{\Z}{{\mathbb Z}}

\newcommand{\C}{{\mathbb C}}
\newcommand{\Q}{{\mathbb Q}}

\newcommand{\ep}{{\varepsilon}}

\title{On a unified double zeta function of Mordell--Tornheim type}

\author[S.~Kadota]{Shin-ya Kadota}
\address[S.~Kadota]{Faculty of Fundamental Science\\
	National Institute of Technology (KOSEN), Niihama College\\
	Niihama-city, Ehime, 792-8580\\
	Japan}
\email{s.kadota@niihama-nct.ac.jp}

\author[T.~Okamoto]{Takuya Okamoto}
\address[T.~Okamoto]{
	Institute of Liberal Arts and Science\\
	Toyohashi University of Technology\\
	Tempaku-cho, Aichi, 441-8580\\
	Japan}
\email{okamoto.takuya.ze@tut.jp}

\author[M.~Ono]{Masataka Ono}
\address[M.~Ono]{
	Global Education Center\\
	Waseda University\\
	Shinjuku, Tokyo, 169-8050\\
	Japan}
\email{m-ono@aoni.waseda.jp}

\author[K.~Tasaka]{Koji Tasaka}
\address[K.~Tasaka]{Department of Information Science and Technology\\
	Aichi Prefectural University\\
	Nagakute-city, Aichi, 480-1198\\
	Japan
}
\email{tasaka@ist.aichi-pu.ac.jp}



\begin{document}

\begin{abstract}
We consider {a} double zeta function of Mordell--Tornheim type and compute its values at non-positive integer points.
We then discuss a possible generalization of the Kaneko--Zagier conjecture for all integer points.
\end{abstract}

\maketitle


\section{Introduction}

The double zeta function of Mordell--Tornheim type is defined by 
\[ \zeta_{MT}(s_1,s_2;s_3)=\sum_{m,n\ge1} m^{-s_1}n^{-s_2}(m+n)^{-s_3},\]
which converges absolutely when $\Re (s_1+s_3),\Re (s_2+s_3)>1$ and $\Re (s_1+s_2+s_3)>2$ (see \cite[Theorem 2.2]{OkamotoOnozuka15}).
The special values of this function at positive integer points are first studied by Tornheim \cite{Tornheim50} and independently by Mordell \cite{Mordell58} for the case $s_1=s_2=s_3$, and also rediscovered by Witten \cite{Witten91} in his volume formula for certain moduli spaces related to theoretical physics (see also Zagier's number theoretical treatment \cite{Zagier94}).
As a function, Matsumoto \cite[Theorem 1]{Matsumoto02} proves that the function $\zeta_{MT}(s_1,s_2;s_3)$ can be analytically continued to the whole $\C^3$-space.
Its true singularities are also determined (see also \cite[Theorem 6.1]{MatsumotoNakamuraOchiaiTsumura08}).
As an application, it is clarified that non-positive integers $s_1,s_2,s_3$ are points of indeterminancy, i.e., the values of $\zeta_{MT}(s_1,s_2;s_3)$ at non-positive integer points depend on a limiting process.
Explicit formulas for these values in terms of generalized Bernoulli numbers are computed by Komori \cite[Theorem 3]{Komori08}.

In this paper, we wish to study special values of the {function}
\[\omega_{\mathcal{U}}(s_1,s_2,s_3)=(-1)^{ s_3}\zeta_{MT}(s_1,s_2;s_3)+(-1)^{ s_2}\zeta_{MT}(s_3,s_1;s_2)+(-1)^{ s_1}\zeta_{MT}(s_2,s_3;s_1),\]
where {$s_1,s_2,s_3$ are complex variables and set} $(-1)^s=e^{\pi i s}$.
The function $\omega_{\mathcal{U}}(s_1,s_2,s_3)$ originates from the previous work in \cite{BachmannTakeyamaTasaka}, where they introduce the values of $\omega_{\mathcal{U}}(s_1,s_2,s_3)$ at positive integer points as a Mordell--Tornheim type analogue of the symmetric multiple zeta values (see also \cite{OnoSekiYamamoto21}). 
One of main results of this paper is an explicit evaluation for {coordinatewise limits} of $\omega_{\mathcal{U}}(s_1,s_2,s_3)$ at non-positive integer points.

\begin{theorem}\label{thm:main}
For any non-negative integers $m_1,m_2,m_3\in \Z_{\ge0}$ {and $a,b,c\in\{1,2,3\}$ such that $\{a,b,c\}=\{1,2,3\}$}, we have
\[ {\lim_{\ep_a\rightarrow0}\lim_{\ep_b\rightarrow0}\lim_{\ep_c\rightarrow0}}\omega_{\mathcal{U}}(-m_1+\ep_{1},-m_2+\ep_{2},-m_3+\ep_{3}) = \begin{cases} 1 & (m_1,m_2,m_3)=(0,0,0)\\ 0 & (m_1,m_2,m_3)\neq (0,0,0)\end{cases} .\]
\end{theorem}

 The contents of this paper are as follows.
In order to motivate our work, we briefly recall the Kaneko--Zagier conjecture on finite/symmetric multiple zeta values in \S2.
\S3 gives a result on true singularities of the function $\omega_{\mathcal{U}}(s_1,s_2,s_3)$.
\S4 is devoted to proving Theorem \ref{thm:main}.
Finally, in \S5, we discuss a possible generalization of the Kaneko--Zagier conjecture for the multiple zeta function of Mordell--Tornheim type.

\section{Background}

This paper is to mimic the study of the unified multiple zeta function introduced by Komori \cite{Komori21}.
We briefly review his results to motivate our work.

For each $r\ge0$, we call a tuple $\boldsymbol{k}=(k_1,\ldots,k_r)$ of positive integers an index and $k_1+\cdots+k_r$ the weight.
We regard the empty index $\emptyset$ as the unique index of weight 0 (and $r=0$).
We set $F(\emptyset)$ to be a unit element for any function $F$ on indices.

We first recall the Kaneko--Zagier conjecture \cite{KanekoZagier}.
For each index $\boldsymbol{k}=(k_1,\ldots,k_r)$ and positive integer $n$, define the multiple harmonic sum $H_n(\boldsymbol{k})$ by
\[ H_n(\boldsymbol{k}) = \sum_{1\le n_1<\cdots<n_r\le n} \frac{1}{n_1^{k_1}\cdots n_r^{k_r}}.\]
The empty sum is understood as 0.
If the last component $k_r$ is greater than 1, the limit at $n\rightarrow \infty$ exists and is called the multiple zeta value, denoted by
\begin{equation}\label{eq:MZV}
\zeta( \boldsymbol{k})=\sum_{1\le n_1<\cdots<n_r} \frac{1}{n_1^{k_1}\cdots n_r^{k_r}}.
\end{equation}
Kaneko and Zagier introduce two different types of multiple zeta values constructed from the above objects.
One of them is the finite multiple zeta value $\zeta_{\mathcal{A}}(\boldsymbol{k})$ defined for each index $\boldsymbol{k}$ by
\[\zeta_{\mathcal{A}}(\boldsymbol{k})=\big(H_{p-1}(\boldsymbol{k})\mod p\big)_p\]
in the $\Q$-algebra $\mathcal{A}=\big(\prod_{p}\Z/p\Z\big) \big/ \big(\bigoplus_p \Z/p\Z\big)$, where $p$ runs over all primes.
Another is the symmetric multiple zeta value $\zeta_{\mathcal{S}}(\boldsymbol{k})$ defined for each index $\boldsymbol{k}$ by
\[ \zeta_{\mathcal{S}}(\boldsymbol{k}) = \sum_{a=0}^r (-1)^{k_{a+1}+\cdots+k_r}\zeta^\shuffle(k_1,\ldots,k_a)\zeta^\shuffle(k_r,\ldots,k_{a+1})\mod \pi^2\mathcal{Z}\]
in the $\Q$-algebra $\mathcal{Z}/\pi^2\mathcal{Z}$, where we denote by $\mathcal{Z}$ the $\Q$-algebra generated by all multiple zeta values and by $\zeta^\shuffle(k_1,\ldots,k_r)\in\mathcal{Z}$ the shuffle regularized multiple zeta value.
Let $\mathbb{I}^+$ be the set of all indices.
The main conjecture of Kaneko and Zagier states that for a finite subset $\{a_{\boldsymbol{k}}\in\Q\mid \boldsymbol{k}\in \mathbb{I}^+\}$ of $\Q$, we have
\begin{equation}\label{eq:KanekoZagierConjecture}
\sum_{\boldsymbol{k}\in \mathbb{I}^+} a_{\boldsymbol{k}} \zeta_{\mathcal{S}}(\boldsymbol{k})=0 \ \stackrel{?}{\Longleftrightarrow}  \ \sum_{\boldsymbol{k}\in \mathbb{I}^+} a_{\boldsymbol{k}} \zeta_{\mathcal{A}}(\boldsymbol{k})=0.
\end{equation}

In \cite{Komori21}, Komori defines the {function} $\zeta_{\mathcal{U}}(s_1,\ldots,s_r)$ by
\[ \zeta_{\mathcal{U}}(s_1,\ldots,s_r)=\sum_{a=0}^r (-1)^{s_{a+1}+\cdots+s_r}\zeta(s_1,\ldots,s_a)\zeta(s_r,\ldots,s_{a+1}).\]
It is a function analogue of symmetric multiple zeta values.
A crucial feature is that the function $\zeta_{\mathcal{U}}(s_1,\ldots,s_r)$ is entire.
Thus, one can study the values $\zeta_{\mathcal{U}}(\boldsymbol{k})\in \C$ for any tuple $\boldsymbol{k}$ of integers.
These values are conjecturally elements in the polynomial ring $\mathcal{Z}[\pi i]$ over $\mathcal{Z}$.
Indeed, for all index $\boldsymbol{k}\in \mathbb{I}^+$, Komori \cite[Theorem 1.4]{Komori21} shows that $\zeta_{\mathcal{U}}(\boldsymbol{k})\in \mathcal{Z}[\pi i]$, and also that 
\begin{equation}\label{eq:zeta_U_cong}
\zeta_{\mathcal{U}}(\boldsymbol{k})\equiv \zeta_{\mathcal{S}}(\boldsymbol{k}) \mod \pi i \mathcal{Z}[\pi i].
\end{equation}

Now, assuming the conjecture $\zeta_{\mathcal{U}}(\boldsymbol{k})\stackrel{?}{\in} \mathcal{Z}[\pi i]$, one extends the definition of the symmetric multiple zeta values $\zeta_{\mathcal{S}}(\boldsymbol{k})$ to any tuple $\boldsymbol{k}$ of integers by
\[\zeta_{\mathcal{S}}(\boldsymbol{k})=\zeta_{\mathcal{U}}(\boldsymbol{k})  \mod \pi i \mathcal{Z}[\pi i].\]
Hereafter, we denote by $\mathbb{I}$ the set of all tuple $(k_1,\ldots,k_r)\in \Z^r \ (r\ge0)$ of integers.
A natural question to ask is then whether the conjecture \eqref{eq:KanekoZagierConjecture} holds for all tuple of integers.
Namely, for a finite subset $\{a_{\boldsymbol{k}}\in\Q\mid \boldsymbol{k}\in \mathbb{I}\}$ of $\Q$, we have
\begin{equation}\label{eq:KomoriConjecture}
\sum_{\boldsymbol{k}\in \mathbb{I}} a_{\boldsymbol{k}} \zeta_{\mathcal{S}}(\boldsymbol{k})=0 \ \stackrel{?}{\Longleftrightarrow}  \ \sum_{\boldsymbol{k}\in \mathbb{I}} a_{\boldsymbol{k}} \zeta_{\mathcal{A}}(\boldsymbol{k})=0.
\end{equation}
Note that, since there is no convergence issue, the finite multiple zeta values $\zeta_{\mathcal{A}}(\boldsymbol{k})$ are already defined for all $\boldsymbol{k}\in \mathbb{I}$ and that they span the same space with the $\Q$-vector space generated by the set $\{\zeta_{\mathcal{A}}(\boldsymbol{k})\mid \boldsymbol{k}\in \mathbb{I}^+\}$ (see \cite{Kaneko19}).
The main result of Komori \cite[Theorem 1.5]{Komori21} states that the conjecture \eqref{eq:KomoriConjecture} holds when we restrict $\boldsymbol{k}$ to tuples of non-positive integers.
A partial evidence for depth 2 case is also given in \cite[Proposition 1.6]{Komori21}.

In this paper we wish to study the Kaneko--Zagier conjecture \eqref{eq:KanekoZagierConjecture} for multiple zeta values of Mordell--Tornheim type, which are independently established in \cite{BachmannTakeyamaTasaka} and \cite{OnoSekiYamamoto21}.
Our terminology follows \cite{BachmannTakeyamaTasaka}. 
In this analogy, for each $\boldsymbol{k}=(k_1,\ldots,k_r)\in \mathbb{I}^+$, let 
\[ \omega_n(\boldsymbol{k})=\sum_{\substack{n_1+\cdots+n_r=n\\n_1,\ldots,n_r\ge1}} \frac{1}{n_1^{k_1}\cdots n_r^{k_r}}\]
and define
\[ \omega_{\mathcal{A}}(\boldsymbol{k}) = \big( \omega_p(\boldsymbol{k}) \mod p\big)_p \in \mathcal{A}\]
for $r\neq1$.
We call this the finite multiple omega value for short.
Note that the finite multiple omega value is, up to sign, equal to the finite multiple zeta value of Mordell--Tornheim type, which is first studied by Kamano \cite{Kamano16}.
As a real counterpart of this, we define the symmetric multiple omega value $\omega_{\mathcal{S}}(\boldsymbol{k})$ for $\boldsymbol{k}=(k_1,\ldots,k_r)\in \mathbb{I}^+$ with $r\neq1$ by
\[\omega_{\mathcal{S}}(\boldsymbol{k})=\sum_{a=1}^r (-1)^{k_a}\zeta_{MT}(\underbrace{k_1,\ldots,k_{a-1}}_{a-1},\underbrace{k_{a+1},\ldots,k_r}_{r-a};k_a)\in \mathcal{Z}/\pi^2\mathcal{Z},\]
where 
\begin{equation}\label{eq:MT}
\zeta_{MT}(k_1,\ldots,k_{r-1};k_r)=\sum_{n_1,\ldots,n_{r-1}\ge1} \frac{1}{n_1^{k_1}\cdots n_{r-1}^{k_{r-1}}(n_1+\cdots+n_{r-1})^{k_r}}
\end{equation}
are the multiple zeta values of Mordell--Tornheim type which can be written as $\Q$-linear combinations of multiple zeta values \eqref{eq:MZV} (see \cite[Theorem 1.1]{BradleyZhou10}).
A similar object to $\omega_{\mathcal{S}}(\boldsymbol{k})$ is introduced by a different method in \cite[Definition 4.5]{OnoSekiYamamoto21} (which in the case $\mathcal{S}$ differs in a sign).
The main conjecture of \cite[Conjecture 1.5]{BachmannTakeyamaTasaka} (see also \cite[\S4]{OnoSekiYamamoto21}) is that for a finite subset $\{a_{\boldsymbol{k}}\in\Q\mid \boldsymbol{k}\in \mathbb{I}^+, r\neq1\}$ of $\Q$, we have
\begin{equation}\label{eq:BTTConjecture}
\sum_{\boldsymbol{k}\in \mathbb{I}^+} a_{\boldsymbol{k}} \omega_{\mathcal{S}}(\boldsymbol{k})=0 \ \stackrel{?}{\Longleftrightarrow}  \ \sum_{\boldsymbol{k}\in \mathbb{I}^+} a_{\boldsymbol{k}} \omega_{\mathcal{A}}(\boldsymbol{k})=0.
\end{equation}
It can be seen from \cite[Theorem 1.6]{BachmannTakeyamaTasaka} that the conjecture \eqref{eq:BTTConjecture} is supported by the Kaneko--Zagier conjecture \eqref{eq:KanekoZagierConjecture}.

An ultimate goal of this project is to establish a Mordell--Tornheim analogue of Komori's conjecture \eqref{eq:KomoriConjecture}.
For that purpose, we need to find a suitable {function} whose values at integer points modulo $\pi^2$ satisfy the same relations with the finite multiple omega values.  
As a natural candidate, we consider the following function:
\[\omega_{\mathcal{U}}(s_1,\ldots,s_r)=\sum_{a=1}^r (-1)^{s_a}\zeta_{MT}(\underbrace{s_1,\ldots,s_{a-1}}_{a-1},\underbrace{s_{a+1},\ldots,s_r}_{r-a};s_a).\]
Similarly to \eqref{eq:zeta_U_cong}, the congruence $\omega_{\mathcal{U}}(\boldsymbol{k})\equiv \omega_{\mathcal{S}}(\boldsymbol{k}) \mod \pi^2 \mathcal{Z}$ holds for all $\boldsymbol{k}\in \mathbb{I}^+$, since $\zeta_{MT}(\boldsymbol{k})$ converges absolutely for any $\boldsymbol{k}=(k_1,\ldots,k_r)\in \mathbb{I}^+$ with $r\neq1$.
As a function, we see from \cite[Theorem 1]{Matsumoto03} that the function $\omega_{\mathcal{U}}(s_1,\ldots,s_r)$ can be meromorphically continued to $\C^r$ and its possible singularities are on
\begin{align}
&s_{j_1}+s_{j_2}=1-l \quad (1\le j_1<j_2\le r ,\ l\in \Z_{\ge0}),\\
&s_{j_1}+s_{j_2}+s_{j_3}=2-l \quad (1\le j_1<j_2<j_3\le r ,\ l\in \Z_{\ge0}),\\
&\vdots\\
&s_{j_1}+\cdots+s_{j_{r-1}}=r-2-l \quad (1\le j_1<\cdots <j_{r-1}\le r ,\ l\in \Z_{\ge0}),\\
&s_1+\cdots+s_r=r-1.
\end{align}
While the {function} $\zeta_{\mathcal{U}}(s_1,\ldots,s_r)$ cancels all possible singularities which come from singularities of multiple zeta functions, some of the above singularities of $\omega_{\mathcal{U}}(s_1,\ldots,s_r)$ will be true.
Thus the generalization of the conjecture \eqref{eq:BTTConjecture} via the {function $\omega_{\mathcal{U}}(s_1,\ldots,s_r)$} is not straightforward.
In this paper, as a part of the future project, we explicate the situation for the case $r=3$ and then discuss a possible generalization of \eqref{eq:BTTConjecture} to all integer points.

\section{Singularities}

By \cite[Theorem 1]{Matsumoto02}, the true singularities of the function $\zeta_{MT}(s_1,s_2;s_3)$ lie on 
$s_1+s_3=1-l,\ s_2+s_3=1-l$ ($l\in\Z_{\ge0}$) and $s_1+s_2+s_3=2$. 
Hence the possible singularities of $\omega_{\mathcal{U}}(s_1,s_2,s_3)$ lie on the subsets of $\C^3$ defined by one of the equations:
\[ s_{a}+s_{b}=1-l\ \ (1\leq a<b\leq 3,\ l\in\Z_{\ge0}),\quad s_1+s_2+s_3=2.\]
We now prove that these are true singularities.

\begin{theorem}\label{thm:singularities}
All points on the subsets of $\C^3$ defined by one of the equations $s_{a}+s_{b}=1-l$\ \ $(1\leq a<b\leq 3,\ l\in\Z_{\ge0})$ and $s_1+s_2+s_3=2$ are true singularities of the function $\omega_{\mathcal{U}}(s_1,s_2,s_3)$.
\end{theorem}

\begin{proof}
We first recall the relevant material from \cite{Matsumoto02}.
For $s_1,s_2,s_3\in \C$, $M\in \Z_{>0}$ and $0<\eta<1$, let
\[I(s_1,s_2,s_3;M-\eta)=\frac{1}{2\pi i} \int_{(M-\eta)} \Gamma(s_3+z)\Gamma(-z)\zeta(s_1+s_3+z)\zeta(s_2-z)\,dz,\]
where the path of integration is the vertical line from $M-\eta-i\infty$ to $M-\eta+i\infty$.
From \cite[Eq.~(5.3)]{Matsumoto02}, we have
\begin{align}\label{MBI}
\begin{split}
\zeta_{MT}&(s_1,s_2;s_3)=\frac{\Gamma(s_2+s_3-1)\Gamma(1-s_2)}{\Gamma(s_3)}\zeta(s_1+s_2+s_3-1)\\
&\quad+\sum_{k=0}^{M-1}\binom{-s_3}{k}\zeta(s_1+s_3+k)\zeta(s_2-k)+\frac{1}{\Gamma(s_3)}I(s_1,s_2,s_3;M-\eta),
\end{split}
\end{align}
where $M$ is a sufficiently large positive integer. We find that the integral $I(s_1,s_2,s_3;M-\eta)$ is holomorphic on the region
$$
D_M=\{ (s_1,s_2,s_3)\in\C^3 \mid \ \Re s_3>-M+\eta, \Re (s_1+s_3)>1-M+\eta, \Re s_2<1+M-\eta\},
$$
because in $D_M$ the poles of the integrand are not on the path of integration. Since $M$ is arbitrarily large, by \eqref{MBI}, the function $\zeta_{MT}(s_1,s_2;s_3)$ can be meromorphically continued to $\C^3$. We note that, in the case $s_2=l$ for a positive integer $l$, the first and the second terms on the right-hand side of \eqref{MBI} are singular, where $M> l$, but these singularities cancel each other. Indeed, put $s_2=l+\ep$ for a positive integer $l$. 
Using the well-known asymptotic formula
\begin{equation}\label{eq:gamma}
\Gamma(s)=
\displaystyle\frac{(-1)^{-n}}{(-n)!}\frac{1}{s-n}+\gamma_{n}+O(|s-n|)\quad (n\in\Z_{\le0}),
\end{equation}
where $\gamma_n$ is the constant term of the expansion of $\Gamma(s)$ at $s=n \in \Z_{\le0}$, we have
\begin{align}
\zeta_{MT}(s_1,s_2;s_3)
&=\frac{\Gamma(s_3+l-1+\ep)\Gamma(1-l-\ep)}{\Gamma(s_3)}\zeta(s_1+s_3+l-1+\ep)\\
&\quad+\binom{-s_3}{l-1}\zeta(s_1+s_3+l-1)\zeta(1+\ep)+O(1)\\
&=\frac{\Gamma(s_3+l-1)}{\Gamma(s_3)}\frac{(-1)^{l-1}}{(l-1)!}\zeta(s_1+s_3+l-1)\frac{1}{\ep}\\
&\quad+\binom{-s_3}{l-1}\zeta(s_1+s_3+l-1)\frac{1}{\ep}+O(1)\\
&=O(1).
\end{align}

We now prove that $s_1+s_3=1-l \ (l\in \Z_{\ge0})$ determines the true singularity of the function $\omega_{\mathcal{U}}(s_1,s_2,s_3)$. 
Take $M\ge l+1$.
From \eqref{MBI} the singular part of $\zeta_{MT}(s_1,s_2;s_3)$ corresponding to $s_1+s_3=1-l$ comes from 
\[{\binom{-s_3}{l}}\zeta(s_1+s_3+l)\zeta(s_2-l).\]
{Since $\zeta_{MT}(s_2,s_3;s_1)=\zeta_{MT}(s_3,s_2;s_1)$}, the corresponding singular part of $\omega_{\mathcal{U}}(s_1,s_2,s_3)$ is
\begin{align}
(-1)^{ s_1}\binom{-s_1}{l}&\zeta(s_1+s_3+l)\zeta(s_2-l)+(-1)^{ s_3}\binom{-s_3}{l}\zeta(s_1+s_3+l)\zeta(s_2-l)\\
&=\left\{(-1)^{ s_1}\binom{-s_1}{l}+(-1)^{ s_3}\binom{-s_3}{l}\right\}\zeta(s_1+s_3+l)\zeta(s_2-l)\\
&=\left((-1)^{ s_1}-(-1)^{ -s_1}\right)\binom{-s_1}{l}\zeta(s_1+s_3+l)\zeta(s_2-l).
\end{align}
Since 
$$
\left((-1)^{ s_1}-(-1)^{ -s_1}\right)\binom{-s_1}{l}\zeta(s_2-l)\not\equiv 0,
$$
we see that $s_1+s_3=1-l$ determines the true singularity. 
Note that we have $\omega_{\mathcal{U}}(s_1,s_2,s_3)=\omega_{\mathcal{U}}(s_{\sigma(1)},s_{\sigma(2)},s_{\sigma(3)})$ for any permutation $\sigma\in \mathfrak{S}_3$.
Thus, $s_{a}+s_{b}=1-l\ \ (1\leq a<b\leq 3,\ l\in\Z_{\ge0})$ determine the true singularities.
 
For the case $s_1+s_2+s_3=2$, by \eqref{MBI} the corresponding singular part of $\omega_{\mathcal{U}}(s_1,s_2,s_3)$ comes from 
$$
\zeta(s_1+s_2+s_3-1)\times \sum_{j=0}^2 (-1)^{ s_{3+j}}\frac{\Gamma(s_{2+j}+s_{3+j}-1)\Gamma(1-s_{2+j})}{\Gamma(s_{3+j})},
$$
where we set $s_{a+j}=s_b$ if $a+j\equiv b \mod 3$ for $b=1,2,3$ and $a,j\in \Z$.
Hereafter, we abuse the same notation for other triplets of complex numbers or variables.
Since $s_2+s_3=2-s_1$, by Euler's reflection formula 
$$
\Gamma(s)\Gamma(1-s)=\frac{\pi}{\sin\pi s},
$$
we obtain 
\begin{align}
(-1)^{ s_{3}}\frac{\Gamma(s_{2}+s_{3}-1)\Gamma(1-s_{2})}{\Gamma(s_{3})} &=(-1)^{s_3}\frac{\Gamma(1-s_1)\Gamma(1-s_{2})}{\Gamma(s_{3})} \\
& = \frac{\pi^2}{\Gamma(s_1)\Gamma(s_2)\Gamma(s_3)} \frac{(-1)^{s_3}}{\sin \pi s_1 \sin \pi s_2}\\
&= \frac{\pi^2}{\prod_{a=1}^3 \Gamma(s_a)\sin \pi s_a } (-1)^{s_3}\sin \pi s_3\\
&= \frac{\pi^2}{\prod_{a=1}^3 \Gamma(s_a)\sin \pi s_a }\frac{(-1)^{2s_3}-1}{2i}.
\end{align}
Since 
$$
\frac{\pi^2}{2i\prod_{a=1}^3 \Gamma(s_a)\sin \pi s_a }\sum_{j=0}^2 \big( (-1)^{2s_{3+j}}-1\big)  \not\equiv 0
$$
holds for the case $s_1+s_2+s_3=2$, we conclude that $s_1+s_2+s_3=2$ determines the true singularity of $\omega_{\mathcal{U}}(s_1,s_2,s_3)$, which completes the proof.
\end{proof}

Our proof of Theorem \ref{thm:singularities} shows that if $(k_1,k_2,k_3)\in \Z^3$ is a singular point of the function $\omega_{\mathcal{U}}(s_1,s_2,s_3)$, then it is a point of indeterminancy.

\begin{remark}
Desingularizations of the double zeta function of Mordell--Tornheim type are studied in \cite{FurushoKomoriMatsumotoTsumura17,Nakamura21}.
\end{remark}

\section{Proof of Theorem \ref{thm:main}}

In this section, we first study an asymptotic behavior of the function $\zeta_{MT}(s_1,s_2;s_3)$ at $(s_1,s_2,s_3)=(-m_1,-m_2,-m_3)$ for non-negative integers $m_1, m_2, m_3$, and then, give a proof of Theorem \ref{thm:main}.

For $m_1,m_2,m_3\in \Z_{\ge0}$, let $M=2+m_1+m_3$.
Then $(-m_1,-m_2,-m_3)\in D_{M}$, so the integral $I(s_1,s_2,s_3;M-\eta)$ is analytic at $(-m_1,-m_2,-m_3)$. Hence we have the estimation
\begin{equation}\label{eq:I}
 \frac{1}{\Gamma(s_3)}I(s_1,s_2,s_3;M-\eta)= O(|s_3+m_3|).
\end{equation}
Note that a similar argument can be found in \cite[\S5]{MatsumotoOnozukaWakabayashi19}.

We now give the asymptotics of the function $\zeta_{MT}(s_1,s_2;s_3)$.
Let $B_k$ be the $k$-th Seki--Bernoulli number defined by
\[
\frac{te^t}{e^t-1}=\sum_{k\ge0}B_k\frac{t^k}{k!}.
\]
For non-negative integers $m_1, m_2, m_3$, let $m=m_1+m_2+m_3+2$ and define rational numbers $b(m_1,m_2;m_3)$ and $c(m_1,m_2;m_3)$ by
\begin{align}
b(m_1, m_2; m_3)&=m_1!m_2!m_3!\binom{m-1}{m_3}\frac{B_{m}}{m!},\\
c(m_1, m_2; m_3)&=\sum_{\substack{n_1+n_2=m_3\\n_1, n_2\ge0}}\binom{m_3}{n_1}\frac{B_{m_1+n_1+1}}{m_1+n_1+1}\frac{B_{m_2+n_2+1}}{m_2+n_2+1}.
\end{align}

\begin{lemma}\label{lem:MT_asymp}
For $m_1,m_2,m_3\in \Z_{\ge0}$ {and $a,b,c\in\{1,2,3\}$ such that $\{a,b,c\}=\{1,2,3\}$, we have
\begin{align}
&\lim_{\varepsilon_a\rightarrow 0}\lim_{\varepsilon_b\rightarrow 0}\lim_{\varepsilon_c\rightarrow 0}\zeta_{MT}(-m_1+\varepsilon_1,-m_2+\varepsilon_2;-m_3+\varepsilon_3)\\
&=\lim_{\varepsilon_a\rightarrow 0}\lim_{\varepsilon_b\rightarrow 0}\lim_{\varepsilon_c\rightarrow 0}\left\{ (-1)^{m_2}b(m_2, m_3; m_1)\frac{\varepsilon_3}{\varepsilon_2+\varepsilon_3}+(-1)^{m_1}b(m_3, m_1; m_2)\frac{\varepsilon_3}{\varepsilon_1+\varepsilon_3}+c(m_1, m_2; m_3)\right\}.
\end{align}}
\end{lemma}

\begin{proof}
Let $m=m_1+m_2+m_3+2$.
For the first term on the right side of \eqref{MBI}, we use \eqref{eq:gamma} and the expansion $\Gamma(s)=(n-1)!+O(|s-n|)$ at $s=n\in\Z_{\ge1}$ to obtain 
\begin{align}
&\frac{\Gamma(-m_2-m_3-1+\ep_2+\ep_3)\Gamma(1+m_2-\ep_2)}{\Gamma(-m_3+\ep_3)}\zeta(1-m+\ep_1+\ep_2+\ep_3)\\
&=\left\{(-1)^{m_3}m_3!\ep_3+O(|\ep_3|^2)\right\}\left\{\frac{(-1)^{m_2+m_3+1}}{(m_2+m_3+1)!}\frac{1}{\ep_2+\ep_3}+\gamma_{-m_2-m_3-1}+O(|\ep_2+\ep_3|)\right\}\\
&\quad\times\left\{m_2!+O(|\ep_2|)\right\}\left\{\zeta(1-m)+O(|\ep_1+\ep_2+\ep_3|)\right\}\\
&=(-1)^{m_2}b(m_2, m_3; m_1)\frac{\ep_3}{\ep_2+\ep_3}\\
&\quad+O\left(\left|\frac{\ep_2\ep_3}{\ep_2+\ep_3}\right|\right)+O\left(\left|\frac{\ep_3(\ep_1+\ep_2+\ep_3)}{\ep_2+\ep_3}\right|\right)+O(|\ep_3|)+O\left(\left|\frac{\ep_3^2}{\ep_2+\ep_3}\right|\right)\\
\end{align}
Here, for the last equality we have used $\zeta(1-m)=-B_m/m \ (m\in \Z_{\ge1})$.

For the second term on the right side of \eqref{MBI}, let $M=2+m_1+m_3$.
Since 
\begin{equation}\label{eq:binomial}
\binom{m-\ep}{k}=
\begin{cases}
\displaystyle\binom{m}{k}+O(|\ep|)&(m\ge k),\\
\displaystyle\frac{(-1)^{k-m}m!(k-m-1)!}{k!}\ep+O(|\ep|^2)&(m<k)
\end{cases}
\end{equation}
holds for $k, m\in\Z_{\ge0}$, we have
\begin{align}
&\sum_{k=0}^{M-1}\binom{m_3-\ep_3}{k}\zeta(-m_1-m_3+k+\ep_1+\ep_3)\zeta(-m_2-k+\ep_2)\\
&=\sum_{k=0}^{m_1+m_3}\binom{m_3-\ep_3}{k}\zeta(-m_1-m_3+k+\ep_1+\ep_3)\zeta(-m_2-k+\ep_2)\\
&\quad+\binom{m_3-\ep_3}{m_1+m_3+1}\zeta(1+\ep_1+\ep_3)\zeta(1-m+\ep_2)\\
&=c(m_1, m_2; m_3)+(-1)^{m_1}b(m_3, m_1; m_2)\frac{\ep_3}{\ep_1+\ep_3}\\
&\quad+O(|\ep_3|)+O(|\ep_1+\ep_3|)+O(|\ep_2|)+O\left(\left|\frac{\ep_3^2}{\ep_1+\ep_3}\right|\right)+O\left(\left|\frac{\ep_2\ep_3}{\ep_1+\ep_3}\right|\right).
\end{align}
Here, for the last equality we have also used the Laurent expansion of $\zeta(s)$ at $s=1$. Thus by \eqref{eq:I}, we obtain the desired result.
\end{proof}

In order to prove Theorem \ref{thm:main}, we use generating functions.
Let
\begin{align}
B(t_1, t_2, t_3)&=\sum_{m_1, m_2, m_3\ge0}(-1)^{m_1+m_2}b(m_1, m_2; m_3)\frac{t_1^{m_1}t_2^{m_2}t_3^{m_3}}{m_1!m_2!m_3!},\\
C(t_1, t_2, t_3)&=\sum_{m_1, m_2, m_3\ge0}(-1)^{m_3}c(m_1, m_2; m_3)\frac{t_1^{m_1}t_2^{m_2}t_3^{m_3}}{m_1!m_2!m_3!}.
\end{align}

\begin{lemma}\label{lem:gen}
We have
\begin{equation}\label{eq:gen_func} 
\sum_{j=0}^2 \big(B(t_{1+j}, t_{2+j}, t_{3+j})+C(t_{1+j}, t_{2+j}, t_{3+j})\big)=1,
\end{equation}
where we again put $t_{a+j}=t_b$ if $a+j\equiv b \mod 3$ for $b=1,2,3$ and $a,j\in \Z$.
\end{lemma}
\begin{proof}
Set
\[\beta(t)=\sum_{k\ge0}\frac{B_{k+1}}{(k+1)!}t^k=\frac{e^t}{e^t-1}-\frac{1}{t}.\]
One computes
\begin{align}
B(t_1, t_2, t_3)&=\sum_{ l, m_3\ge0}\left(\sum_{\substack{m_1+m_2=l\\m_1, m_2\ge0}}t_1^{m_1}t_2^{m_2}\right)(-1)^l\binom{l+m_3+1}{m_3}\frac{B_{l+m_3+2}}{(l+m_3+2)!}t_3^{m_3}\\
&=\sum_{l, m_3\ge0}\frac{t_1^{l+1}-t_2^{l+1}}{t_1-t_2}(-1)^l\binom{l+m_3+1}{m_3}\frac{B_{l+m_3+2}}{(l+m_3+2)!}t_3^{m_3}\\
&=\frac{1}{t_1-t_2}\sum_{k\ge1}\left(\sum_{\substack{l+m_3+1=k\\l, m_3\ge0}}(-1)^{l+1}\binom{k}{m_3}(t_2^{l+1}-t_1^{l+1})t_3^{m_3}\right)\frac{B_{k+1}}{(k+1)!}\\
&=\frac{1}{t_1-t_2}\sum_{k\ge1}\left((t_3-t_2)^k-t_3^k - (t_3-t_1)^k+t_3^k\right)\frac{B_{k+1}}{(k+1)!}\\
&=\frac{\beta(t_3-t_2)-\beta(t_3-t_1)}{t_1-t_2},
\end{align}
and
\begin{align}
C(t_1, t_2, t_3)&=\sum_{m_1, m_2, n_1, n_2\ge0}\frac{B_{m_1+n_1+1}}{m_1+n_1+1}\frac{B_{m_2+n_2+1}}{m_2+n_2+1}\frac{t_1^{m_1}t_2^{m_2}(-t_3)^{n_1+n_2}}{m_1!m_2!n_1!n_2!}\\
&=\prod_{j=1}^2 \left(\sum_{m_j, n_j\ge0}\frac{B_{m_j+n_j+1}}{(m_j+n_j+1)!} \binom{m_j+n_j}{m_j}t_j^{m_j}(-t_3)^{n_j}\right)\\
&=\beta(t_1-t_3)\beta(t_2-t_3).
\end{align}
Substituting these expressions into the left side of \eqref{eq:gen_func} and then using
\[ \sum_{j=0}^2 \frac{1}{(t_{1+j}-t_{2+j})(t_{3+j}-t_{2+j})} =0\quad \mbox{and} \quad \sum_{j=0}^2 \frac{e^{t_{1+j}-t_{3+j}}}{e^{t_{1+j}-t_{3+j}}-1} \frac{e^{t_{2+j}-t_{3+j}}}{e^{t_{2+j}-t_{3+j}}-1} =1, \]
we get the desired result.
\end{proof}

We are now in a position to prove Theorem \ref{thm:main}.

\begin{proof}[Proof of Theorem \ref{thm:main}]
By Lemma \ref{lem:MT_asymp}, we have
\begin{align}
&{\lim_{\ep_a\rightarrow0}\lim_{\ep_b\rightarrow0}\lim_{\ep_c\rightarrow0}}\omega_{\mathcal{U}}(-m_1+\ep_1, -m_2+\ep_2, -m_3+\ep_3)\\
&=\sum_{j=0}^2 (-1)^{m_{1+j}+m_{2+j}}b(m_{1+j}, m_{2+j}; m_{3+j}){\lim_{\ep_a\rightarrow0}\lim_{\ep_b\rightarrow0}\lim_{\ep_c\rightarrow0}}\frac{(-1)^{\ep_{1+j}}\ep_{1+j}+(-1)^{\ep_{2+j}}\ep_{2+j}}{\ep_{1+j}+\ep_{2+j}}\\
&\quad +\sum_{j=0}^2{ (-1)^{m_{3+j}}c(m_{1+j}, m_{2+j}; m_{3+j})}.
\end{align}
Since ${\displaystyle\lim_{\ep_a\rightarrow0}\lim_{\ep_b\rightarrow0}\lim_{\ep_c\rightarrow0}}\frac{(-1)^{\ep_{1+j}}\ep_{1+j}+(-1)^{\ep_{2+j}}\ep_{2+j}}{\ep_{1+j}+\ep_{2+j}}=1$, we only need to prove 
\begin{align}\label{eq:goal}
&\sum_{j=0}^2 (-1)^{m_{1+j}+m_{2+j}}b(m_{1+j}, m_{2+j}; m_{3+j})+\sum_{j=0}^2 (-1)^{m_{3+j}}c(m_{1+j}, m_{2+j}; m_{3+j})\\
&= \begin{cases} 1 & (m_1,m_2,m_3)=(0,0,0),\\ 0 & (m_1,m_2,m_3)\neq (0,0,0).\end{cases}
\end{align}
Since the left side of \eqref{eq:goal} coincides with the coefficient of $\frac{t_1^{m_1}t_2^{m_2}t_3^{m_3}}{m_1!m_2!m_3!}$ in the left side of \eqref{eq:gen_func}, the desired result follows from Lemma \ref{lem:gen}
\end{proof}

\section{Conclusion}

 By Theorem \ref{thm:main}, one can define the symmetric double omega values $\omega_{\mathcal{S}}(k_1,k_2,k_3)$ for non-positive integers $k_1,k_2,k_3\in \Z_{\le0}$ by
\begin{equation}\label{eq:limit_omega}
\omega_{\mathcal{S}}(k_1,k_2,k_3)=
{\lim_{\ep_a\rightarrow0}\lim_{\ep_b\rightarrow0}\lim_{\ep_c\rightarrow0}}\omega_{\mathcal{U}}(k_1+\ep_1,k_2+\ep_2,k_3+\ep_3)\mod \pi^2\mathcal{Z},
\end{equation}
which are either 0 or 1.
On the other hand, for $k_1,k_2,k_3\in \Z_{\le0}$ one computes
\begin{align}
\omega_{\mathcal{A}}(k_1,k_2,k_3)&=\left(\sum_{\substack{n_1+n_2<p\\n_1,n_2\ge1}} n_1^{-k_1}n_2^{-k_2}(p-n_1-n_2)^{-k_3}\mod p\right)_p\\
&=\left( \sum_{m_2=2}^{p-1}\sum_{m_1=1}^{m_2-1}m_1^{-k_1}(m_2-m_1)^{-k_2}(-m_2)^{-k_3}\mod p\right)_p\\
&=(-1)^{-k_2-k_3}\sum_{l=0}^{-k_2} (-1)^l \binom{-k_2}{l} \zeta_{\mathcal{A}}(k_1+k_2+l,k_3-l)
\end{align}
in $\mathcal{A}$.
Hence by \cite[Theorem 1.3]{Komori18}, for non-positive integers $k_1,k_2,k_3\in \Z_{\le0}$ we have
\[ \omega_{\mathcal{A}}(k_1,k_2,k_3)=\begin{cases} (1)_p & (k_1,k_2,k_3)=(0,0,0)\\ (0)_p & (k_1,k_2,k_3)\neq (0,0,0)\end{cases}.\]
Therefore the conjecture \eqref{eq:BTTConjecture} holds when we restrict $\boldsymbol{k}$ to elements in $\Z^3_{\le0}$.
Moreover, for integers $k_1,k_2,k_3$ such that $k_1\le k_2=0\le k_3$, as a consequence of \cite[Proposition 1.6]{Komori21}, we obtain the correspondence between the limit value \eqref{eq:limit_omega} and $\omega_{\mathcal{A}}(k_1,k_2,k_3)$, because we have $\omega_{\mathcal{U}}(k_1,0,k_3)=(-1)^{k_3}\zeta_{\mathcal{U}}(k_1,k_3)$ and $\omega_{\mathcal{A}}(k_1,0,k_3)=(-1)^{k_3}\zeta_{\mathcal{A}}(k_1,k_3)$.

Now a possible analogue of \eqref{eq:KomoriConjecture} for multiple zeta values of Mordell--Tornheim type is as follows.
Suppose that a certain limit ${\displaystyle\lim_{\boldsymbol{\ep}\rightarrow \boldsymbol{0}}}'\ \omega_{\mathcal{U}}(\boldsymbol{k}+\boldsymbol{\ep})$ exists in $\mathcal{Z}$ for all $\boldsymbol{k}\in \mathbb{I}$.
Then one can define the symmetric multiple omega values $\omega_{\mathcal{S}}(\boldsymbol{k})\in \mathcal{Z}/\pi^2\mathcal{Z}$ for all $\boldsymbol{k}\in \mathbb{I}$ by
\[\omega_{\mathcal{S}}(\boldsymbol{k})={\lim_{\boldsymbol{\ep}\rightarrow \boldsymbol{0}}}'\ \omega_{\mathcal{U}}(\boldsymbol{k}+\boldsymbol{\ep}) \mod \pi^2 \mathcal{Z}.\] 
One would ask if
\[\sum_{\boldsymbol{k}\in \mathbb{I}} a_{\boldsymbol{k}} \omega_{\mathcal{S}}(\boldsymbol{k})=0 \ \stackrel{?}{\Longleftrightarrow}  \ \sum_{\boldsymbol{k}\in \mathbb{I}} a_{\boldsymbol{k}} \omega_{\mathcal{A}}(\boldsymbol{k})=0\]
holds for a finite subset $\{a_{\boldsymbol{k}}\in\Q\mid \boldsymbol{k}\in \mathbb{I}, r\neq1\}$ of $\Q$.


\section*{Acknowledgments}
This work is partially supported by
JSPS KAKENHI Grant Number JP16H06336 and 20K14294.
{The authors thank Yasushi Komori for pointing out factual mistakes in the early paper}.


\end{document}